\documentclass{amsart}
\usepackage{graphics}
\usepackage{amsfonts,amsmath}
\usepackage{amsmath, amsfonts,amssymb,amsopn,amscd,amsthm}
 \usepackage{hyperref}
 \usepackage{mathrsfs}
 \hypersetup{colorlinks=true,linkcolor=blue,citecolor=blue,linktoc=page}

\begin{document}

 \newtheorem{thm}{Theorem}[section]
 \newtheorem{cor}[thm]{Corollary}
 \newtheorem{lem}[thm]{Lemma}{\rm}
 \newtheorem{prop}[thm]{Proposition}

 \newtheorem{defn}[thm]{Definition}{\rm}
 \newtheorem{assumption}[thm]{Assumption}
 \newtheorem{rem}[thm]{Remark}
 \newtheorem{ex}{Example}
\numberwithin{equation}{section}

\def\e{{\rm e}}
\def\x{\boldsymbol{x}}
\def\c{\boldsymbol{c}}
\def\b{\boldsymbol{b}}
\def\blambda{\boldsymbol{\lambda}}
\def\P{\mathbf{P}}
\def\I{\mathbf{I}}
\def\Q{\mathbf{Q}}
\def\B{\mathbf{B}}
\def\N{\mathbb{N}}
\def\by{\mathbf{y}}
\def\z{\boldsymbol{z}}
\def\R{\mathbb{R}}
\def\A{\mathbf{A}}
\def\U{\mathbf{U}}
\def\C{\mathbf{C}}
\def\Tr{\text{Tr}}
\def\G{\mathbf{G}}
\def\X{\mathbf{X}}
\def\L{\mathbf{L}}
\def\f{\mathbf{f}}
\def\y{\boldsymbol{y}}
\def\bv{\varepsilon}
\def\bI{\mathbf{I}}
\def\y{\mathbf{y}}
\def\g{\mathbf{g}}
\def\w{\mathbf{w}}
\def\b{\mathbf{b}}
\def\a{\mathbf{a}}
\def\u{\mathbf{u}}
\def\q{\mathbf{q}}
\def\e{\mathbf{e}}
\def\s{\mathcal{S}}

\def\om{\mathbf{\Omega}}
\def\blue{\color{blue}}
\def\red{\color{red}}
\def\blambda{\boldsymbol{\lambda}}
\def\btheta{\boldsymbol{\theta}}
\def\balpha{\boldsymbol{\alpha}}
\def\dis{\displaystyle}
\def\1{\boldsymbol{1}}
\def\va{\vert\balpha\vert}
\title[Entropy-regularization of LP and SDP]
{Shannon- and Von Neumann-entropy regularizations of Linear and Semidefinite Programs}
\author{S.P. Chhatoi \and JB. Lasserre}
\address{LAAS-CNRS, 7 avenue du Colonel Roche\\
31077 Toulouse C\'edex 4, France}
\email{spchhtoi@laas.fr}
\address{LAAS-CNRS and Toulouse School of Economics (TSE)\\
LAAS, 7 avenue du Colonel Roche\\
31077 Toulouse C\'edex 4, France}
\email{lasserre@laas.fr}

\date{}

\maketitle
\begin{abstract}
We consider the LP in standard form $\min\,\{\c^T\x: \A\x=\b;\x\geq0\}$
and inspired by $\bv$-entropy regularization in Optimal Transport, 
we introduce its $\varepsilon$-regularization 
$\min\,\{\c^T\x+\varepsilon\,f(\x): \A\x=\b;\x\geq0\}$ 
via the (convex) Boltzmann-Shannon entropy $f(\x):=\sum_i\x_i\ln\x_i$. We also provide 
a similar regularization for the 
semidefinite program $\min\,\{\Tr(\C\cdot\X): \mathcal{A}(\X)=\b;\X\succeq0\}$ 
but with now the so-called Von Neumann entropy, as in Quantum Optimal Transport. Importantly, both entropies are not barriers of the LP and SDP cones respectively.
%Invoking a Lagrangian relaxation, w
We show that for both LP and SDP, the resulting regularized problem admits an equivalent unconstrained convex problem $\max_{\blambda\in\R^m} G_\varepsilon(\blambda)$ for an explicit 
concave differentiable function $G_\varepsilon$ 
in dual variables $\blambda\in\R^m$. As $\varepsilon$ goes to zero, its optimal value converges to the optimal value of the initial LP. In the limit, one also recovers the unique optimal solution of the LP that minimizes the Shanon entropy among all LP-minimizers. 
While it resembles the log-barrier formulation of interior point algorithms, 
it has a distinguishing 
advantage. Namely, $G_\varepsilon(\blambda)$ is obtained in closed form (via an explicit minimizer $\x(\blambda)>0$) for every arbitrary $\blambda\in\R^m$, hence with \emph{no} need to restrict $\blambda$ to the feasible set of the dual LP.
This explicit form of $G_\varepsilon$ is crucial for its maximization over the whole $\R^m$. 
\end{abstract}

\section{Introduction}

In this paper we introduce the Boltzmann-Shannon 
entropy-regularization of linear programs (LPs) and the von Neumann entropy-regularization of semidefinite programs (SDPs), and describe and comment some of their nice features. The main remarkable feature is that with $\varepsilon>0$ fixed, arbitrary, in both LP and SDP cases, one has to solve an unconstrained convex optimization problem whose reward function is explicit, strictly concave and differentiable. Hence basic optimization algorithms can be applied. 

Namely, Let $\P$ and $\Q$ be a linear program (LP) 
and $\Q$ be a semidefinite program (SDP) in standard form:
\begin{eqnarray}
    \label{def-LP}
 \P:\quad   \tau&=&\min\,\{\,\c^T\x: \A\x=\b\,;\:\x\geq0\,\}\\
 \label{def-SDP}
    \Q:\quad \rho&=&\min\,\{\,\Tr(\C\cdot\X) : \mathcal{A}(\X)\,=\b\,;\:\X\succeq0\,\},
    \end{eqnarray}
where $\c\in\R^d$, $\C\in\mathcal{S}^n$, $\b\in\R^m$, $\A\in\R^{m\times d}$, and $\mathcal{A}:\mathcal{S}^n\to \R^m$ is a linear map (with $\mathcal{S}^n$ being the space of real $(n\times n)$ symmetric matrices).

Their respective $\varepsilon$-regularization $\P_\varepsilon$ and $\Q_\varepsilon$ (for $\varepsilon>0$) via Shannon- and von Neumann-entropy\footnote{Observe that $\Tr(\ln(\X))$ is the analogue for positive semidefinite matrices of $\sum_i\ln(\x_i)$ for $\x\geq0$. Indeed $\Tr(\ln(\X))=\sum_i\ln(\sigma_i)$ where 
$(\sigma_i)_i$ are the eigenvalues of $\X$.}, read:
\begin{eqnarray}
    \label{def-LP-epsilon}
   \P_\varepsilon:\: \tau_\varepsilon&=&
   \min_{\x}\,\{\,\c^T\x+\varepsilon\,\sum_{i=1}^d\x_i\,\ln (\x_i): \A\x=\b\,;\:\x\geq0\,\}\\
   \label{def-SDP-epsilon}
    \Q_\varepsilon:\:\rho_\varepsilon&=&\min_{\X}\,\{\,\Tr(\C\cdot\X)
    +\varepsilon \,\Tr(\X\cdot\ln(\X)): \mathcal{A}(\X)\,=\b\,;\:\X\succeq0\,\}\,.
    \end{eqnarray}
    
    %Similarly, with $\mathcal{S}^n$ the space of real $(n\times n)$ symmetric matrices, consider the semidefinite program: 
%\begin{equation}
    %\label{def-SDP}
    %\Q:\quad \rho\,=\,\min\,\{\,\Tr(\C\cdot\X) : \mathcal{A}(\X)\,=\b\,;\:\X\succeq0\,\},
    %\end{equation}
%where $\C=\C^T\in\R^{n\times n}$, $\b\in\R^m$, and $\mathcal{A}:\mathcal{S}^n\to \R^m$ is a linear map.
%Its $\varepsilon$-regularization (for $\varepsilon>0$) 
%In fact $\X\mapsto \Tr(\X\cdot\ln(\X))$ is called the 
%von Neumann entropy\footnote{Observe that $\Tr(\ln(\X))$ is the analogue for positive semidefinite matrices of $\sum_i\ln(\x_i)$ for $\x\geq0$. Indeed $\Tr(\ln(\X))=\sum_i\sigma_i$ where 
%$(\sigma_i)_i$ are the eigenvalues of $\X$.}
%reads:
%\begin{equation}    \label{def-SDP-epsilon}
    %\Q_\varepsilon:\quad\rho_\varepsilon\,=\,\min\,\{\,\Tr(\C\cdot\X)
    %+\varepsilon \,\Tr(\X\cdot\ln \X): \mathcal{A}(\X)\,=\b\,;\:\X\succeq0\,\},
%\end{equation}
Notice that $\P_\varepsilon$ resembles the log-barrier 
method 
\begin{equation}
\label{def-logbarrier-LP}
    \min_{\x}\,\{\,\c^T\x-\varepsilon\,\sum_{i=1}^d\ln(\x_i): \A\x=\b\,;\:\x\geq0\,\}\,\end{equation}
to solve \eqref{def-LP} via interior point methods,
and iterating with a sequence of parameters 
$\varepsilon_k\downarrow 0$ as $k\to\infty$.
Similarly,
the log-barrier method for solving \eqref{def-SDP} reads
\begin{equation}
    \label{def-logbarrier-SDP}
    \min_{\X}\,\{\,\Tr(\C\cdot\X)
    -\varepsilon \,\ln\mathrm{det}(\X): \mathcal{A}(\X)\,=\b\,;\:\X\succeq0\,\}\,.\end{equation}
In fact, the function $-\sum_i\ln(\x_i)$ is also an entropy function initially used by Boltzmann. However in the context of
LP and SDP, and to the best of our knowledge, in its initial version, the log-barrier was not referred to as an entropy regularization. It was more conceived as a ``repulsive barrier" to prevent $\x$ (resp. $\X$) from approaching the boundary of the convex cone $\R^d_+$ (resp. $\mathcal{S}^n_+$). The emphasis was on its self-concordance property of utmost importance for efficiency of interior point methods. More recently, an \emph{entropic} variant was proposed in \cite{bubek} where its properties are related to that of log-concave densities in probability. In the discussion in \cite[Section 2.1]{bubek}
Hildebrand's canonical barrier (a variant) is also related to the differential entropy of a Gaussian probability measure.\\

But we emphasize that in contrast with the log-barrier, the Boltzmann-Shannon entropy $\x\mapsto \sum_i\x_i\ln(\x_i)$ is \emph{not} a barrier for the positive orthant because $u\ln(u)\to 0$ whenever $u\to 0$. This is why the term
$\varepsilon$-regularization is 
more appropriate. Indeed without the entropy (i.e., $\varepsilon=0$), an optimal solution $\x^*(0)\in\R^d$ of the LP is a vertex of the polytope hence with at least $d-m$ zero coordinates. With $\varepsilon>0$ the associated optimal solution $\x^*(\varepsilon)$ has \emph{no} zero coordinate
and is a regularization of the vertex to a positive vector 
with small coordinates $\x^*_i(\varepsilon)$ if
$\x^*(0)=0$ and $\varepsilon>0$ is small. Similarly for SDPs,
$\X^*(0)$ is not full rank whereas $\X^*(\varepsilon)$
is  full rank, but with some very small eigenvalues if $\varepsilon>0$ is small.

Concerning SDPs, note that the classical log-barrier $-\ln(\mathrm{det}(\X))$ is also not comparable with the Von Neumann entropy $\Tr(\X\cdot\ln(\X))$ in \eqref{def-SDP-epsilon}, as again and as in the scalar case, the latter is \emph{not} a barrier for the SDP cone.

Depending on the context (thermodynamics, information theory, probability, etc..) the entropy $\sum_i\x_i\ln\x_i$ is given different names, like e.g., Boltzmann, Gibbs, Shannon). 
For an account on and historical references to entropy,
the interested reader is referred to e.g., \cite{gerolin-1,gerolin-2,Tryphon,Villani} and the many references therein.

To the best of our knowledge, we are not aware of prior works using the Boltzmann-Shannon entropy to solve or approximate 
general LPs or SDPs (possibly because they are not \emph{barriers} of the respective LP and SDP cones). An exception is the LP formulation of (commutative) optimal transport and the SDP formulation of non-commutative
quantum optimal transport. The former (resp. the latter) is a particular case of \eqref{def-LP-epsilon} (resp. 
\eqref{def-SDP-epsilon}) and resulted in the celebrated Sinkhorn algorithm. 
%In the latter case the resultingentropy is called the von Neumann entropy. 
For more details on the topic of optimal transport and entropy, the interested reader is referred to e.g. \cite{brenier,beier,Cole,Cuturi,Tryphon,gerolin-1,gerolin-2,Villani}. It is worth noting that the LP formulation 
of optimal transport can be solved by Kuhn's Hungarian method
as a more efficient alternative to the simplex algorithm.

\subsection*{Contribution}
Assume that the feasible set of 
\eqref{def-LP} (resp. \eqref{def-SDP}) is compact and Slater's condition holds for 
\eqref{def-LP} (resp. \eqref{def-SDP}).
Then we prove that for fixed $\varepsilon>0$:

(i) Solving $\P_\varepsilon$ reduces to solving the unconstrained convex optimization problem
\begin{equation}
\label{final-LP}
    \max_{\blambda\in\R^m}\b^T\blambda-\frac{\varepsilon}{e}\sum_{i=1}^d \exp((\A^T\blambda-\c)_i/\varepsilon)\quad(=:\,\max_{\blambda} 
 G_\varepsilon(\blambda))\,,
    \end{equation}
where $G_\varepsilon$ is concave, differentiable, and attains its maximum at some $\blambda^*\in\R^m$ because $-G_\varepsilon$ is coercive.

(ii) Solving $\Q_\varepsilon$ reduces to solving the unconstrained convex optimization problem
\begin{equation}
\label{final-SDP}
    \max_{\blambda\in\R^m}
    \b^T\blambda-\varepsilon\,\Tr(\exp(\mathcal{A}^*\blambda-\C)/\varepsilon)-\bf I)\quad (=:\max_{\blambda\in\R^m} \hat{G}_\varepsilon(\blambda))\,,
    \end{equation}
where $\hat{G}_\varepsilon$ is concave, differentiable, and attains its maximum at some $\blambda^*\in\R^m$, also because $\hat{G}_\varepsilon$ is coercive.\\

(iii) We also provide asymptotics as $\varepsilon\downarrow0$. Namely,
$\tau_\varepsilon\to\tau$ as $\varepsilon\downarrow 0$, and therefore with $\varepsilon>0$ fixed, sufficiently small,
one obtains a close approximation of $\tau$
by solving the single unconstrained convex optimization 
\eqref{final-LP}. 
Similarly, $\rho_\varepsilon\to\rho$ as $\varepsilon\downarrow 0$, and therefore with $\varepsilon>0$ fixed, sufficiently small, one obtains a close approximation of $\rho$
by solving the single unconstrained convex optimization 
\eqref{final-SDP}.

In addition and interestingly, consider a sequence $\x^*(\varepsilon_k)_{k\in\N}$ of optimal solutions of 
$\P_{\varepsilon_k}$, with $\varepsilon_k\downarrow0$ as $k\to\infty$. Then the sequence converges to the unique optimal solution $\x^*$ of $\P$ which minimizes 
the Boltzmann-Shannon entropy among the set of all optimal solutions of $\P$. The same result also holds for a sequence $\X^*(\varepsilon_k)_{k\in\N}$ of optimal solutions of $\Q_{\varepsilon_k}$.\\

(iv) For illustration purposes and also as a proof of concept, we provide 
a (limited) set of numerical experiments on LPs and SDPs of non-trivial size.
For SDPs note that even though \eqref{final-SDP} is an unconstrained problem, evaluation of $\exp((\mathcal{A}^*\blambda-\C)/\varepsilon)$ for a given $\blambda\in\R^m$ requires a singular value decomposition 
$\U^T\L\U$ of $\mathcal{A}^*\blambda-\C$ to exponentiate 
the diagonal elements of $\L$, and therefore as expected, 
maximizing $\hat{G}_\varepsilon$ in \eqref{final-SDP} 
is significantly more difficult\footnote{Notice that in the $\log$-barrier method one has to evaluate 
$\ln(\mathrm{det}(\X))$, a non-trivial task, let alone 
that $\X$ should be maintained to be positive definite.} than maximizing $G_\varepsilon$ in \eqref{final-LP}. 
\vspace{.2cm}

\subsection*{Remarks and discussion}~
\vspace{0.2cm}

$\bullet$ Comparing \eqref{def-LP-epsilon} with \eqref{def-logbarrier-LP}: To solve \eqref{def-logbarrier-LP} one needs
to maintain the semidefinite constraint $\x>0$, that is, 
\eqref{def-logbarrier-LP} is \emph{not} an unconstrained minimization problem. In contrast, \eqref{final-LP}
is the Lagrangian relaxation-scheme applied to $\P_\varepsilon$ and thanks to 
the particular form of the Shannon entropy, 
the unconstrained inner minimization
\[\min_{\x\in\R^d} \c^T\x+\blambda^T(\b-\A\x)+\varepsilon\,f(\x)
%\sum_{i=1}^d \x_i\ln(\x_i)
\quad ( =\,G_\varepsilon(\blambda)-\b^T\blambda)\]
(where $f(\x):=\sum_{i=1}^d \x_i\ln(\x_i)$ if $\x\geq0$ and 
$+\infty$ otherwise) is obtained in closed form with a strictly positive 
minimizer $\x(\varepsilon)$, which yields the concave differentiable function $G_\varepsilon$
in explicit form \eqref{final-LP}.

The Shannon-entropy has been used in $\varepsilon$-regularization of (static) optimal transport (OT), and the characterization of optimal solutions,
has resulted in the Sinkhorn algorithm widely used (e.g., in Machine Learning applications of OT). For more details on
OT and the Sinkhorn algorithm the interested reader is referred to e.g. \cite{Tryphon,Cuturi} and references therein.
However to the best of our knowledge we are not aware of the use of Boltzmann-Shannon entropy to help solve LPs.
\vspace{0.2cm}

$\bullet$  Comparing \eqref{def-SDP-epsilon} with \eqref{def-logbarrier-SDP}: To solve \eqref{def-logbarrier-SDP} with interior point methods, one needs
to maintain the semidefinite constraint $\X\succ0$, that is, 
\eqref{def-logbarrier-SDP} is \emph{not} an unconstrained minimization problem. In contrast, \eqref{final-SDP}
is the Lagrangian relaxation-scheme applied to $\Q_\varepsilon$ and thanks to 
the particular form of the Shannon entropy, 
the unconstrained inner minimization
\[\min_{\X\in\mathcal{S}^n} \Tr(\C\cdot\X)+\blambda\,(\b-\mathcal{A}(\X))+\varepsilon\,g(\X)
%\Tr(\X\cdot\ln(\X))
\quad (=\,\hat{G}_\varepsilon(\blambda)-\b^T\lambda)\]
where $g(\X):=\Tr(\X\cdot\ln(\X))$ if $\X\succ0$,
$0$ if $\X\succeq0$ and $\X$ is singular, and $+\infty$ otherwise)
is obtained in closed form with a strictly positive 
definite minimizer $\X(\varepsilon)$, and yields the function concave differentiable function $\hat{G}_\varepsilon$
in explicit form \eqref{final-SDP}.\\

In fact the entropy function $\X\mapsto \Tr(\X\cdot\ln(\X))
$ is rather called the Von Neumann entropy, and is used 
in some Quantum Optimal Transport problems formulated 
as (complex) SDPs with Hermitian matrices of very large dimension. So while inspired from this approach in Quantum Optimal Transport, our formulation $\Q_\varepsilon$ for general SDPs \eqref{def-SDP} turns out to be in a much more favorable situation, at least for reasonable dimensions.

\section{Main result}
Recall that $\lim_{u\downarrow 0}u\,\ln(u)=0$ and so
by continuity at $0$, consider the convex function
 \[\x\mapsto f(\x)\,:=\,\left\{\begin{array}{rl}\sum_{i=1}^d \x_i\log(\x_i)&\mbox{if $\x\geq0$}\\
 +\infty&\mbox{otherwise,}\end{array}\right.,\]
(with the convention $0\ln(0)=0$) which is convex and continuous on its domain $\mathrm{Dom}(f)\,(=\{\x:\x\geq0\})$,
and differentiable on $\mathrm{int}(\mathrm{Dom}(f))=\{\x>0\}$.

\subsection{Entropy-regularization of LPs}

Consider the linear program (LP) \eqref{def-LP} and its
$\varepsilon$-regularization \eqref{def-LP-epsilon}
where $\c\in\R^d$, $\b\in\R^m$,  and $\A\in\R^{m\times d}$.
For fixed $\varepsilon>0$ define the Lagrangian
\[(\x,\blambda)\mapsto \mathcal{L}_\varepsilon(\x,\blambda)\,:=\,
\c^T\x+\blambda^T(\b-\A\x)+\varepsilon\,f(\x)\,,\quad (\x,\blambda)\in \R^d\times\R^m\,,\]
associated with $\P_\varepsilon$, and let
\[\blambda\mapsto G_\varepsilon(\blambda)\,:=\,\inf_{\x\in\R^d}\mathcal{L}_\varepsilon(\blambda,\x)\,,\quad \blambda\in\R^d\,.\]
\begin{lem}
\label{lem-1}
    Assume that there exists $\blambda_0\in\R^m$ such that
    $\A^T\blambda_0>0$  (so that $\om:=\{\,\x: \A\x=\b;\,\x\geq0\,\}$ is compact).
    %and assume that Slater's condition holds, i.e., there exists $0<\x_0\in\om$. 
    Then for every fixed $\varepsilon>0$, 
    $G_\varepsilon$ is concave, differentiable and satisfies
    \begin{eqnarray}
        \label{lem-1:1}
G_\varepsilon(\blambda)&=&
\mathcal{L}_\varepsilon(\blambda,\x^*(\blambda))\,=\,
\b^T\blambda-\varepsilon\,\sum_{i=1}^d\x^*_i(\blambda)\,,
\quad\forall\blambda\in\R^m\\
        \label{lem-1:2}
        \mbox{with $\x^*_i(\blambda)$}&:=&\exp((\A^T\blambda-\c)_i/\varepsilon -1)\,>\,0\,,\quad i=1,\ldots,d\,;\\
        \label{lem-1:3}
        \nabla G_\varepsilon(\blambda)&=&\b-\A\x^*(\blambda)\,,\quad\forall\blambda\in\R^m\,.
    \end{eqnarray}
 \end{lem}
 \begin{proof}
 The function $\x\mapsto \mathcal{L}_\varepsilon(\blambda,\x)$ is strictly convex and differentiable on $\mathrm{int}(\mathrm{Dom}(f))$. Let $\x^*(\blambda)$ be as in \eqref{lem-1:2} and so
 $\x^*(\blambda)\in\mathrm{int}(\mathrm{Dom}(f))$. Observe that $\nabla_{\x} \mathcal{L}(\x^*(\blambda),\blambda)=0$, which implies that $\x^*(\blambda)$ is the unique minimizer of $\x\mapsto \mathcal{L}_\varepsilon(\blambda,\x)$, and yields \eqref{lem-1:1}-\eqref{lem-1:2}. To get \eqref{lem-1:3}, substitute 
  \eqref{lem-1:2}  in the gradient 
 $\nabla_{\x}G_\varepsilon(\x^*(\blambda))$. Finally,
 concavity of  $G_\varepsilon$ follows substituting \eqref{lem-1:2}  in $\mathcal{L}_\varepsilon(\x^*(\blambda),\blambda)$.
 \end{proof}
\subsection*{Coercivity of $G_\varepsilon$}
Let $\mathbb{S}^{m-1}\subset\R^m$ denote the Euclidean unit sphere, and assume that $\A$ is full row rank. 
\begin{lem}
\label{lem:coercive}
   Assume that there exists $0<\x_0\in\om$ and let $G_\varepsilon$ be as in \eqref{final-LP}. Then $-G_\varepsilon$ is coercive, that is, $G_\varepsilon\to-\infty$ as $\Vert\blambda\Vert\to +\infty.$
\end{lem}
\begin{proof}
    We first prove that if a sequence $(\blambda_k)_{k\in\N}$
is such that $\Vert\b^T\blambda\Vert\to\infty$ as $k\to\infty$,
then there exists $i$ such that $\vert(\A^T\blambda_k)_i\vert\to  +\infty$ as $k\to\infty$. Suppose not, that is, $\sup_k\sup_i\vert(\A^T\blambda_k)_i\vert< a$ for some scalar $a>0$. Then $\blambda_k/\Vert\blambda_k\Vert\in \mathbb{S}^{m-1}$,
and by compactness, there exists a subsequence
relabelled $(\blambda_k)_{k\in\N}$ for convenience, such that
$\blambda_k/\Vert\blambda_k\Vert\to\btheta\in \mathbb{S}^{m-1}$.
Hence for every $i=1,\ldots,d$, $\vert(\A^T\blambda_k/\Vert\blambda_k\Vert)_i\vert<a/\Vert\blambda_k\Vert\to 0$ as $k\to\infty$. Therefore
$\A^T\btheta=0$ which in turns implies $\btheta=0$ (as $\A$ is full row rank), in contradiction with $\Vert\btheta\Vert=1$.

Next, write $\b=\A\x_0$ (with $\x_0>0$), so that 
$\b^T\blambda= (\A^T\blambda-\c)^T\x_0+\c^T\x_0$.
Then for a sequence $(\blambda_k)_{k\in\N}$ with
$\Vert\blambda_k\Vert\to\infty$,
\begin{eqnarray*}
  G_\varepsilon(\blambda_k)&=&\b^T\blambda_k-\varepsilon\sum_{i=1}^d \x_i(\blambda_k)\\
  &=&
  \c^T\x_0+(\A^T\blambda_k-\c)^T\x_0-\varepsilon\Vert\x(\blambda_k)\Vert_1\\
    &=&\c^T\x_0+\sum_{i=1}^d (\A^T\blambda_k-\c)_i\left[(\x_0)_i-\varepsilon\frac{\exp((\A^T\blambda-\c)_i/\varepsilon)}{e\,(\A^T\blambda_k-\c)_i}
    \right]\,.
    \end{eqnarray*}
 So let $(\Vert\blambda_k\Vert)\to\infty$, and let 
 $I:=\{i: \vert(\A^T\blambda)_i\vert\to+\infty\}$ so that $I\neq\emptyset$.   
If $i\in I$ and $(\A^T\blambda_k-\c)_i\to-\infty$ then the item $i$
of the bracket is nonnegative  (as $(\x_0)_i>0)$ and 
therefore the item $i$ of the sum goes to $-\infty$.
Similarly, if $(\A^T\blambda-\c)_i\to+\infty$ then in the item $i$ the bracket goes to $-\infty$ and therefore 
the item $i$ of the sum goes to $-\infty$. Therefore
the contribution in sum of every item $i\in I$ gives $-\infty$, which yields the desired result.
\end{proof}  
\begin{cor}
\label{cor-1}
For every $\varepsilon>0$, the regularization $\P_\varepsilon$ of $\P$ has a unique optimal solution $0<\x^*(\varepsilon)\in\om$, and under the assumption of Lemma \ref{lem:coercive},
\begin{eqnarray}
    \label{cor-1:1} \tau_\varepsilon&=&\min_{\x\geq0}\c^T\x+\varepsilon\,f(\x)\,=\,\c^T\x^*(\varepsilon)+\varepsilon\sum_{i=1}^d \x^*_i(\varepsilon)\,\ln(\x^*_i(\varepsilon))\\
    \label{cor-1:2}
    &=&\max_{\blambda\in\R^m}G_\varepsilon(\blambda)\,=\,\b^T\blambda^*-\varepsilon\,\sum_{i=1}^d\x^*_i(\varepsilon)\,.
\end{eqnarray}
 \end{cor}
    
\begin{proof}
 $\P_\varepsilon$ is equivalent to $\min_{\x\in\om} \c^T x +\varepsilon f(\x)$, and the infimum of a continuous function on a compact set is attained. Moreover as $\x\mapsto \c^T\x+\varepsilon\,f(\x)$ is strictly convex, a minimizer $\x^*(\varepsilon)\in\om$ is unique.

 Next, recall that $G_\varepsilon(\blambda)\leq \tau_\varepsilon$ for all $\blambda\in\R^m$, $G_\varepsilon$ is continuous, differentiable and by Lemma \ref{lem:coercive}, $-G_\varepsilon$ is coercive. Therefore its maximum is attained at $\blambda^*\in\R^m$ where $\nabla G_\varepsilon(\blambda^*)=0=\b-\A\,\x^*(\blambda^*)$, and so
 $\x^*(\blambda^*)\in\om$. Moreover a minimizer $\blambda^*$ is unique because in view of \eqref{lem-1:1}, and
 \eqref{lem-1:2}, $G_\varepsilon$ is strictly concave. Finally,  as 
 \[G_\varepsilon(\blambda^*)=\L_\varepsilon(\blambda^*,\x^*(\blambda^*))\,=\,\c^T\x^*(\blambda^*)+\varepsilon\,f(\x^*(\blambda^*))\,\leq\,\tau_\varepsilon\,,\]
 and $\x^*(\blambda^*)\in\om$, we conclude that $\x^*(\blambda^*)$ is a minimizer of $\P_\varepsilon$ and by uniqueness 
 $\x^*(\blambda^*)=\x^*(\varepsilon)$.
 \end{proof}
 
As a result, one may obtain $\tau_{\bv}$ by solving an unconstrained optimization problem \eqref{cor-1:1}, i.e.,
by maximizing the concave (and explicit) differentiable function 
$G_{\bv}$ over $\R^d$. Therefore with $\epsilon>0$ small and fixed, one obtains an approximation
of $\tau$, the optimal value of the LP \eqref{def-LP}. Moreover, at the unique maximizer $\blambda^*$ of \eqref{cor-1:1}, one extracts the unique minimizer $\x^*(\varepsilon)=\x^*(\blambda^*))$
of $\P_\varepsilon$ directly by the formula \eqref{lem-1:2}.

Finally, observe that in the canonical form \eqref{def-LP}, one has $m\leq d$
and even sometimes $m\ll d$, which is an additional nice feature of \eqref{cor-1:1}. 

\subsection{Comparison with the log-barrier}
Let $\mu>0$ be fixed, and consider
\[\P_\mu:\quad\tau_\mu\,:=\,\inf_{\x} \{\,\c^T\x-\mu\sum_{i=1}^d \ln(\x_i):\:\A\x=\b\,\}\,.\]
At a minimizer $\x(\mu)$, the necessary KKT-optimality conditions yield
\[\A\x(\mu)\,=\,\b\,;\quad (\c-\A^T\blambda(\mu))_i=\mu/\x_i(\mu):,,\quad i=1,\ldots,d\,,\]
for some multiplier $\blambda(\mu)$. That is:
\[\x_i(\mu)\,=\,\frac{\mu}{\c_i-\A^T\blambda(\mu)}\,,\quad i=1,\ldots,d\,,\]
from which nonnegativity is not automatic. It requires that 
$\A^T\blambda(\mu)<\c$, and so the analogue $H_\mu$ of $G_\varepsilon$ reads:
\[H_\mu(\blambda)\,:=\,\b^T\blambda+\sum_{i=1}^d \frac{\c_i\mu}{(\c-\A^T\blambda)_i}+\mu\ln ((\c-\A^T\blambda)_i)\,,\quad \A^T\blambda\,<\,\c\,,\]
defined only for $\A^T\blambda<\c$, whereas $G_\varepsilon$
is defined for every $\blambda\in\R^m$. This is an important difference if one wants to maximize 
$H_\mu$ to retrieve the optimal value of $\P_\mu$ via Lagrangian relaxation-scheme\footnote{However, in interior point algorithms one usually directly minimizes
$\c^T\x-\mu\sum_{i=1}^d \ln(\x_i)$ over the subspace $\A\x=\b$, and then one iterates with a lower value of the parameter $\mu$, etc.}. Finally one remarks that 
if $f(\x)$ is \emph{not} a barrier for the positive orthant,
its gradient $\x\mapsto \sum_{i=1}^d\ln(\x_i)+1$ 
acts as  a barrier in gradient methods.

\subsection{Semidefinite Programming via von-Neumann entropy}
Consider the semidefinite program (SDP)
\begin{equation}
\label{SDP}
\P:\quad \rho\,=\,\min\,\{\Tr (\C\cdot\X):\: \mathcal{A}(\X)\,=\b\,;\:\X\succeq 0\, \}
\end{equation} 
where $\C\in \mathcal{S}^n$ and $\mathcal{A}\,:\,\mathcal{S}^n\mapsto\R^m$ is a linear mapping.

We assume that
the spectrahedron $\om:=\{\,\mathcal{A}(\X)\,=\b\,;\:\X\succeq 0\,\}$ is compact, and we
consider the following  regularization $\P_{\bv}$ of $\P$, where $\bv>0$ is fixed, arbitrary.
\begin{equation}
\label{SDPreg}
\Q_{\bv}:\quad \rho_{\bv}\,=\,\min_\X\,\{\Tr (\C\cdot\X)\,+\, \bv\,\Tr(\X\cdot\ln(\X)):\: \mathcal{A}(\X)\,=\,\b\,;\:\X\succ 0\, \}.
\end{equation} 
$\P_{\bv}$ is a convex problem as the von-Neumann entropy 
\[   \X\mapsto \Tr(\X\cdot\ln(\X))\,,\quad\forall \X\in\mathcal{S}^n_{++}\:(:=\{\,\X\in\mathcal{S}^n: \X\succ0\,\})\]
is a convex function. Next define
\begin{equation}
    \label{def-von-neumann}
    \X\mapsto g(\X):=\left\{\begin{array}{rl}\Tr(\X\cdot\ln(\X)) &\mbox{if $\X\succ0$}\\
    0 & \mbox{if $\X \succeq 0$ and $\X$ is singular} \\
    +\infty &\mbox{otherwise.}\end{array}\right.
\end{equation}
So $g$ is continuous on $\mathcal{S}^n_+=\{\X: \X\succeq0\}$ by using the convention $u\ln(u)=0$ if $u=0$.
\begin{prop}
The convex optimization problem has an optimal solution.
\end{prop}
\begin{proof}
It follows from the fact that $\om$ is compact and 
in \eqref{SDPreg} one may rather consider the cost function $\X\mapsto \Tr(\C\cdot\X)+\varepsilon \,g(\X)$ which is continuous.
\end{proof}
Let $\mathcal{A}^*:\R^m\to\mathcal{S}^n$ be the adjoint linear, injective mapping associated with $\mathcal{A}$, and introduce the Lagrangian $G_\varepsilon:\R^m\to\R$:
\begin{equation}
\label{def-G-sdp}
G_\varepsilon(\blambda)\,:=\,\b^T\blambda+
\inf_\X\,\{\Tr (\X\cdot(\C\,-\,\mathcal{A}^*\blambda))+\,\bv\,\Tr(\X\cdot\ln(\X))\,\}\,.\end{equation}
Observe that $G_\varepsilon(\blambda)\leq\rho_\varepsilon$, for every $\blambda\in\R^m$.
\begin{lem}
\label{lem:G}
For every fixed $\varepsilon>0$, the function $G_\varepsilon$ is strictly concave and satisfies
\begin{eqnarray}
        \label{lem:G-formula}
    G_\varepsilon(\blambda)&=&\b^T\blambda-\varepsilon\,\mathrm{Tr} \,\X(\blambda)\\
    \nonumber\mbox{with }\X(\blambda)&=&\exp((\mathcal{A}^*\blambda-\C)/\varepsilon)-\mathbf{I})\,.
    \end{eqnarray}
\end{lem}
\begin{proof}
 The function
 \[\X\mapsto \Tr ((\C\,-\,\mathcal{A}^*\blambda)\cdot \X)+\,\bv\,\Tr(\X\cdot\log(\X))\]
 is convex. Its gradient, which reads
 \[\X\mapsto \C-\mathcal{A}^*\blambda+\varepsilon\,(\ln(\X)+\mathbf{I})\,,\quad \forall\X\succ0\,,\]
 vanishes at $\X(\blambda):=\exp((\mathcal{A}^*\blambda-\C)/\varepsilon-\mathbf{I})\succ0$.  
 Next, evaluating
 $G_\varepsilon$ at $\X(\blambda)$ 
 yields  \eqref{lem:G-formula}, from 
 which strict concavity follows.
 \end{proof}

\subsection*{Coercivity of $G_\varepsilon$}
We first recall the following technical result.
\begin{prop}
\label{prop-eigenvalue}
  Let $\A,\B\in\mathcal{S}^n$ be $(n,n)$ real symmetric matrices, with $\B\succeq0$, and let $(\lambda^A_j)$
  and $(\lambda^B_j$ be their respective eigenvalues arranged in increasing order $\lambda^A_1\leq\cdots\leq\lambda^A_n$,
  and $\lambda^B_1\leq\cdots\leq\lambda^B_n$. Then:
  \begin{equation}
      \label{aux-1}
      \sum_{j=1}^n \lambda^A_j\,\lambda^B_{n-j-1}\,\leq\,\Tr(\A\cdot\B)\,\leq\,\sum_{j=1}^n \lambda^A_j\,\lambda^B_{j}\,.
  \end{equation}
\end{prop}
\begin{proof}
 If $\A\succeq0$ the result is well-known, otherwise, write
 $\A=\A+a\,\bf I$ -$a\,\bf I$ for $a>0$ sufficiently large so that
 $\lambda^{\A+a\,\bf I}_j=\lambda^A_j+a\geq0$ for all $j$.
\end{proof}

\begin{lem}\label{lem:SDPcoercive}
   Let $\om$ be compact and assume that:
   
   -- Slater's condition holds, i.e., there exists $0\prec\X_0\in \om$, and

   -- $\mathcal{A}^*\blambda=0\Rightarrow \blambda=0$.

   Then $G_\varepsilon(\blambda)\to -\infty$ whenever $\Vert\blambda\Vert\to+\infty$,
\end{lem}
\begin{proof}
We first prove that $\Vert\mathcal{A}^*\blambda\Vert\to\infty$ whenever $\Vert\blambda\Vert\to\infty$. So let $(\blambda_k)_{k\in\N}$ be a sequence such that $\Vert\blambda_k\Vert\to\infty$ as $k\to\infty$.  Then there is a subsequence still denoted $(\blambda_k)_{k\in\N}$ for convenience, such that
$\blambda_k/\Vert\blambda_k\Vert\to \blambda^*\in \mathbb{S}^{m-1}$. Hence suppose that $\sup_k\Vert\mathcal{A}^*\blambda_k\Vert<a$ for some $
a>0$. Then 
$\mathcal{A}^*\blambda_k/\Vert\blambda_k\Vert\to \mathcal{A}^*\blambda^*=0$, which in turn implies $\blambda^*=0$, in contradiction with $\blambda^*\in\mathbb{S}^{m-1}$.
Next write
\[\b^T\blambda\,=\,\Tr(\mathcal{A}^*\blambda-\C)\cdot\X_0)+\Tr(\C\cdot\X_0)\,,\]
so that
\begin{eqnarray*}
  G_\varepsilon(\blambda_k)&=&\b^T\blambda_k-\varepsilon\,\Tr(\X(\blambda_k))\\
  &=&
  \Tr(\C\cdot\X_0)+\Tr((\mathcal{A}^*\blambda_k-\C)\cdot\X_0)-\varepsilon\,\Tr(\X(\blambda_k))\\
  &=&\Tr(\C\cdot\X_0)+\Tr((\mathcal{A}^*\blambda_k-\C)\cdot\X_0)-\varepsilon\,\Tr(\exp(\mathcal{A}^*\blambda_k-\C)/\varepsilon)-\bf I)
    \end{eqnarray*}
Let $(\sigma_1(\blambda_k),\ldots\sigma_m(\blambda_k))$ be the eigenvalues of the real symmetric matrix
$\mathcal{\A}^*\blambda_k-\C$, and let $(u_1,\ldots,u_m)$
that of $\X_0\succ0$, arranged in increasing order. By Proposition \ref{prop-eigenvalue}:
\begin{eqnarray*}
    \Tr((\A^*\blambda_k-\C)\cdot\X_0)&\leq&
\sum_{i=1}^n u_i\,\sigma_i(\blambda_k)\\
\Tr(\exp((\A^*\blambda_k-\C)/\varepsilon-I)&\geq&
\sum_{i=1}^n \exp(\sigma_i(\blambda_k)/\varepsilon-1)\,,
\end{eqnarray*}
and therefore
\[G_\varepsilon(\blambda_k)\,\leq\,\sum_{i=1}^n 
u_i\,\sigma_i(\blambda_k)-\varepsilon
\exp(\sigma_i(\blambda_k)/\varepsilon-1)\,.\]
Since $\Vert\mathcal{A}^*\blambda_k\Vert\to\infty$ as $k\to\infty$, one has $\vert\sigma_i(\lambda_k)\vert\to\infty$ for at least one index $i$, and so let $I^+:=\{i:\sigma_i(\blambda_k)\to+\infty\}$, and  $I^-:=\{i:\sigma_i(\blambda_k)\to-\infty\}$. As $k\to\infty$:
\begin{eqnarray*}
    u_i\,\sigma_i(\blambda_k)-\varepsilon\,
\exp(\sigma_i(\blambda_k)/\varepsilon-1)&\to&-\infty
\mbox{ if $i\in I^+$}\\
\underbrace{u_i}_{>0}\,\sigma_i(\blambda_k)-\varepsilon\,
\exp(\sigma_i(\blambda_k)/\varepsilon-1)&\to&-\infty
\mbox{ if $i\in I^-$.}
\end{eqnarray*}
This shows that $G_\varepsilon\to-\infty$ as $k\to\infty$, and so $-G_\varepsilon$ is coercive.
\end{proof}

\begin{cor}
\label{SDPcor-1}
The regularization $\P_\varepsilon$ of $\P$ has a unique optimal solution $0<\X^*(\varepsilon)\in\om$, and under the assumption of Lemma \ref{lem:SDPcoercive},
\begin{eqnarray}
    \label{SDPcor-1:1} \tau_\varepsilon&=&
        \Tr(\C\cdot\X^*(\varepsilon))\,+\, \bv\,\Tr(\X^*(\varepsilon)\cdot\log(\X^*(\varepsilon)))\\
    \label{SDPcor-1:1}
    &=&\max_{\blambda\in\R^m}G_\varepsilon(\blambda)\,=\,
    \b^T\blambda^*-\varepsilon\,\Tr(\X^*(\varepsilon))\,.
    \end{eqnarray}

 \end{cor}
 The proof follows similar arguments as the linear programming case.
 
\section{Asymptotic Analysis and Computational Experiments}

In this section, we analyze the asymptotic behavior of the solution to the entropy-regularized linear program as $\varepsilon\downarrow 0$, and provide some computational experiments that illustrate the approach.

\subsection{Asymptotic Behavior of the Regularized Solution}
We consider the asymptotic behavior of the solution $\x^*(\varepsilon)$ to the entropy-regularized linear program as the regularization parameter $\varepsilon \downarrow 0$. 
\begin{lem}
    For every $\varepsilon>0$, let $\x^*(\varepsilon)$ be the unique minimizer of \eqref{def-LP-epsilon}. Then $\x^*(\varepsilon)$ converges to the unique minimizer $\x^*$ of the LP \eqref{def-LP} which minimizes the Shannon entropy among all minimizers of \eqref{def-LP}.
\end{lem}
\begin{proof}
Consider a sequence $\varepsilon_k \downarrow 0$, and its associated sequence of unique minimizers $\x^*(\varepsilon_k)$ of $\P_{\varepsilon_k}$ as defined in \eqref{def-LP-epsilon}, and let $\x$ be any minimizer of $\P$ as defined in \eqref{def-LP}. Since $\om$ is compact,  $\x^*(\varepsilon_k) \to \x^*$ (up to a subsequence). Then we have 
 \begin{eqnarray*}
0\le\,\c^T\x^*(\varepsilon_k) - \c^T\x\le \varepsilon_k\,\big(\sum_{i=1}^d\x_i\ln \x_i-\sum_{i=1}^d\x_i^*(\varepsilon_k)\ln \x_i^*(\varepsilon_k) \big) 
 \end{eqnarray*}
where the first inequality follows from the fact that $\x$ is the minimizer of \eqref{def-LP} and the second inequality follows from the fact that $\x^*(\varepsilon)$ is the minimizer of \eqref{def-LP-epsilon}. As $\varepsilon_k \downarrow 0$, we obtain $\c^T\x^* = \c^T \x$. Moreover, dividing $\varepsilon_k$ on the right-hand-side inequality,
one obtains $\sum_{i=1}^d\x_i\ln \x_i \ge \sum_{i=1}^d\x_i^*(\varepsilon_k)\ln \x_i^*(\varepsilon_k)$. Thus taking the limit as $k\to\infty$, one obtains $\sum_{i=1}^d\x_i\ln \x_i \ge \sum_{i=1}^d\x_i^*\ln \x_i^*$. That is, $\x^*$ minimizes the Shannon entropy among all minimizers of \eqref{def-LP}. Since entropy is strictly convex, the whole sequence $\x^*(\varepsilon_k)$ converges to the one that minimizes  Shannon entropy (and so is unique).
\end{proof}

\begin{rem}
    For semidefinite programs, a similar result holds for the minimizers $\X(\blambda)$ of problem $\P_\varepsilon$, as defined in \eqref{SDPreg}, with a similar proof.
\end{rem} 

\subsection{Computational Experiments}
We next provide some numerical results to illustrate the method for solving entropy-regularization of linear programs. To showcase the proposed method,
we have considered two different type of examples.

\subsection*{Large-Scale Linear Program}
Our first example consists of a non trivial LP with $10 000$ variables and $50$ constraints. The matrix $\A$ and the vector $\b$ were generated randomly while ensuring feasibility by setting $\b = \A \x^*$ for a randomly generated positive vector $\x^*$ (so that Slater's condition holds). The cost vector $\c$ was also generated randomly. 
To obtain $\blambda^*$, the dual function $G_{\varepsilon}(\blambda)$ was maximized using MATLAB’s \texttt{fminunc} function based on a quasi-Newton method.

With $\varepsilon:=0.01$, we observed that solving 20 random instances of the regularized LP problem required only an average of 15 iterations to converge, with the norm of the gradient of $\G_\varepsilon$ at optimality being on the order of $1e-4$; recall that the gradient of $\G_\varepsilon$ is the feasibility constraint in \eqref{def-LP}.

\subsection*{Discrete Optimal Transport Problem}
As already mentioned, for optimal transport problems (OT)
(a very specific type of LPs), the celebrated Sinkhorn algorithm (also based on a Shannon-entropy regularization) is widely used; see e.g. \cite{Cuturi}. So 
we have also compared our algorithm with Sinkhorn algorithm
on a toy OT example
with cost matrix
\begin{equation*}
C = \begin{pmatrix}
4 & 1 \\
2 & 3
\end{pmatrix}\,,
\end{equation*}
and with source and target distributions (marginals) :
\begin{equation} \label{marginals}
p = \begin{pmatrix}
0.5 \\
0.5
\end{pmatrix}
\quad \text{and} \quad
q = \begin{pmatrix}
0.6 \\
0.4
\end{pmatrix}.
\end{equation}
A regularization parameter of $\varepsilon = 0.01$ was applied. The dual problem was solved using again \texttt{fminunc}, to yield the optimal dual solution and the corresponding regularized cost. The optimal value $G_\varepsilon(\blambda^*)$ of the regularized LP for the OT problem was \textbf{1.7906} with $\Vert\nabla G_\varepsilon(\blambda^*)\Vert \approx 1e-6$, while the optimal value obtained via the Sinkhorn algorithm\footnote{We refer to the MATLAB implementation by Marco Cuturi,  available at \url{https://github.com/marcocuturi/SinkhornTransport}}
was \textbf{1.7906}. 
\begin{rem}
    The Sinkhorn algorithm is an iterative matrix scaling algorithm used to find a bistochastic matrix by iteratively normalizing its rows and columns. A key property of the Sinkhorn algorithm is the factorization of the exponential map, which characterizes the primal solution as 
\[
\exp(\phi + \psi - C).
\]
where $\phi,\psi$ are lagrange multipliers corresponding to the marginal constraints in the Sinkhorn algorithm. An example of such marginals is provided in \eqref{marginals}.
This factorization enables efficient updates and guarantees geometric convergence at a rate of \(\mathcal{O}(e^{-k/\varepsilon})\), where \(k\) is the iteration count and \(\varepsilon>0\) is the entropic regularization parameter. 

However, such a factorization is not possible in \eqref{lem-1:3}, because the matrix \( \A^T \blambda \) contains mixed terms that prevent the exponential matrix from being factorized. This lack of separability is an obstruction
to envision Sinkhorn-like algorithms, as iterative matrix scaling techniques rely on independent row and column updates, which are not possible in this setting.

\end{rem}
\subsection*{Semidefinite Programming Example}
We also implemented our method for solving problem $\Q_\varepsilon$ defined in \eqref{SDPreg} by maximizing the dual function $G_\varepsilon$  given in \eqref{lem:G-formula}. We solved a problem with $\mathcal{A}: \mathcal{S}^{100} \to \R^{20}$, where the individual matrices were randomly generated symmetric matrices. The problem was solved by again using MATLAB’s \texttt{fminunc} function. In 20 random experiments, we observed that the problem converged on average in 15 iterations.

We wish to emphasize that the numerical results provided in this section are meant to showcase the proposed method. While we have successfully implemented the method for medium to relatively large LPs ({\bf with matrices 
$\A$ of dimension around $(50\times 10 000)$}), for semidefinite programming problems we have encountered numerical issues with large dimensions and very small regularization parameter $\varepsilon$, very likely
due to our use of basic Matlab's matrix exponential routines which are prone to numerical instabilities. 
Further investigation beyond the scope of the present paper,
is required to 
(i) develop robust and faster algorithms to implement the proposed method, and (ii) possibly provide rates of convergence for fixed $\varepsilon$. In addition, a detailed comparison of the proposed method with existing methods is a
topic further investigation.

\begin{rem}
    Table \ref{algo_comparison} below draws a comparison between the proposed method and the Sinkhorn algorithm, similar to how the Simplex algorithm is compared to the Hungarian algorithm for solving linear programming and optimal transport problems, respectively.
\begin{table}[h!]
\centering
\renewcommand{\arraystretch}{2} 
\begin{tabular}{|c|c|c|}
\hline
\textbf{} & \textbf{LP} & \textbf{OT} \\
\hline
\textbf{Classical Algorithms} & Simplex  & Hungarian \\
\hline
\textbf{Entropy-regularization} & Proposed method & Sinkhorn \\
\hline
\end{tabular}
\caption{(LP) and optimal transport (OT) problems,
with their entropy-regularization.}
\label{algo_comparison}
\end{table}
\end{rem}

\section*{Acknowlegements}
This work was done when the two authors were invited at the IPAM institute in UCLA during the long program \emph{Non Commutative Optimal Transport} (Spring 2025). It was inspired by the use of Shannon entropy in (static) OT and benefited from various discussions with participants in that program. The authors gratefully acknowledge financial support from IPAM.

 \end{document}